\documentclass[11pt]{article}

\usepackage{graphicx}
\usepackage{amsmath}
\usepackage{amsthm}
\usepackage{amsfonts}
\usepackage{amssymb, amscd}

\usepackage[all, knot]{xy}

\newtheorem{thm}{Theorem}[section]

\newtheorem{lem}[thm]{Lemma}
\newtheorem{prop}[thm]{Proposition}

\theoremstyle{definition}
\newtheorem{defn}[thm]{Definition}
\theoremstyle{remark}
\newtheorem{rem}[thm]{Remark}
\numberwithin{equation}{section}

\input{epsf.sty}

\begin{document}

\title{Second Cohomology of  $q$-deformed Witt superalgebras}

\author{Faouzi Ammar
\and Abdenacer Makhlouf
\and Nejib Saadaoui }

\author{Faouzi Ammar \\\small{ \emph{Universit\'{e} de Sfax}}\\ \small{Facult\'{e} des Sciences,}\\
\small{B.P. 1171, Sfax 3000, Tunisia}\\ \small{Faouzi.Ammar@rnn.fss.tn}
\and
 Abdenacer Makhlouf \\
 \small{\emph{Universit\'{e} de Haute Alsace},
LMIA,} \\
\small{4, rue des Fr\`{e}res Lumi\`{e}re}\\\small{ F-68093 Mulhouse, France} \\
\small{Abdenacer.Makhlouf@uha.fr }\and Nejib Saadaoui\\\small{  \emph{Universit\'{e} de Sfax}}\\ \small{Facult\'{e} des Sciences,}\\
\small{B.P. 1171, Sfax 3000, Tunisia}\\ \small{najibsaadaoui@yahoo.fr}}

\maketitle

\begin{abstract}
The purpose of this paper is to compute the second adjoint cohomology group of $q$-deformed Witt superalgebras. They are Hom-Lie superalgebras obtained by $q$-deformation of Witt Lie superalgebra, that is one considers $\sigma$-derivations instead of classical derivations.
\end{abstract}

\section*{Introduction}
The theory of Hom-Lie superalgebras was introduced in \cite{AmmarMakhloufJA2010}.
 Representations and a cohomology theories of  Hom-Lie superalgebras was provided  in \cite{AMSa1}. Moreover we have studied central extensions and provide as application  computations of the derivations and scalar second cohomology group  of $q$-deformed Witt superalgebras. In this paper, we aim to provide  computation of second adjoint cohomology group  of  $q$-deformed Witt superalgebras.
 
 The Witt algebra is one of the simplest infinite dimensional Lie algebra. This Lie algebra of vector fields, was defined by E. Cartan. It is established that it admits one central extension, that is Virasoro algebras. The computation of the cohomology and the formal rigidity of Witt and Virasoro algebras was established by Fialowski \cite{Fial1990}, see also \cite{F3,F4,fialowski-schlichenmaier,Gelfand-Fuchs,Guieu-Roger,Schlichenmaier}. For Lie superalgebras we refer to 
\cite{Retach-Feigen,Sheunert1,Sheunert2,Sheunert3}.

In the first Section we review some preliminaries,  the cohomology of Hom-Lie superalgebras and deformation theory. In Section 2, we describe $q$-Witt superalgebras. The main result, about  second adjoint cohomology of $q$-deformed Witt superalgebras, is stated in Section 3. Its proof is given  by computing even and odd  adjoint second cohomology groups.

\section{Preliminaries}

Let  $\mathcal{G}$ be a linear superspace over a field $\mathbb{K}$ that is a $\mathbb{Z}_{2}$-graded linear space with a direct sum $\mathcal{G}=\mathcal{G}_{0}\oplus \mathcal{G}_{1}.$
The elements of $\mathcal{G}_{j}$, $j\in \mathbb{Z}_{2},$ are said to be homogenous  of parity $j.$ The parity of  a homogeneous element $x$ is denoted by $|x|.$
The space $End (\mathcal{G})$ is $\mathbb{Z}_{2}$-graded with a direct sum $End (\mathcal{G})=End _{0}(\mathcal{G})\oplus End _{1}(\mathcal{G})$, where
$End _{j}(\mathcal{G})=\{f\in End (\mathcal{G}) : f (\mathcal{G}_{i})\subset \mathcal{G}_{i+j}\}.$ 
The elements of $End_{j}(\mathcal{G})$  are said to be homogenous of parity $j.$ 
Let $\displaystyle \mathcal{E}=\oplus_{n\in\mathbb{Z}}\mathcal{E}_{n}$ be a $\mathbb{Z}$-graded linear space, a linear map $f \in End (\mathcal{E})$ is called of degree $s$ if $f(\mathcal{E}_{n})\subset\mathcal{E}_{n+s},$ for all $n\in \mathbb{Z}$.

\begin{defn}\cite{AmmarMakhloufJA2010}
A Hom-Lie superalgebra  is a triple $(\mathcal{G},\ [.,.],\ \alpha)$\ consisting of a superspace $\mathcal{G}$, an even bilinear map \ $\ [.,.]:\mathcal{G}\times \mathcal{G}\rightarrow \mathcal{G}$ \ and an even superspace homomorphism \ $ \alpha:\mathcal{G}\rightarrow \mathcal{G} \ $satisfying
\begin{eqnarray}
&&[x,y]=-(-1)^{|x||y|}[y,x],\\
&&(-1)^{|x||z|}[\alpha(x),[y,z]]+(-1)^{|z||y|} [\alpha(z),[x,y]]+(-1)^{|y||x|} [\alpha(y),[z,x]]=0,\label{jacobie}
\end{eqnarray}
for all homogeneous element $x, y, z$ in $\mathcal{G}.$
\end{defn}
\subsection{ Cohomology of  Hom-Lie Superalgebras}
Let $(\mathcal{G},[.,.],\alpha)$ be a Hom-Lie superalgebra and $V=V_{0}\oplus V_{1}$ an arbitrary vector superspace. Let $\beta\in\mathcal{G}l(V)$ be an arbitrary even linear self-map on $V$  and
$ [.,.]_{V}:
\begin{array}{ccc}
\mathcal{G}\times V& \rightarrow& V \\
(g,v)&\mapsto& [g,v]_{V}
\end{array}$ 
 be a bilinear map  satisfying $[\mathcal{G}_{i},V_{j}]_{V}\subset V_{i+j}$ where $i,j\in \mathbb{Z}_{2}.$
\begin{defn}
The triple $(V,[.,.]_{V}, \beta)$  is called a Hom-module on the Hom-Lie superalgebra $\mathcal{G}=\mathcal{G}_{0}\oplus \mathcal{G}_{1}$ or  $\mathcal{G}$-module $V$  if the  even bilinear map $[.,.]_{V}$ satisfies
\begin{eqnarray}
[\alpha(x),\beta(v)]_{V}&=&\beta([x,v]_{V}) \label{rep1}
\\
\left[[x,y],\beta(v)\right]_{V}&=&\left[\alpha(x),[y,v]\right]_{V}-(-1)^{|x||y|}\left[\alpha(y),[x,v]\right]_{V}, \label{rep2}
\label{mod}
\end{eqnarray}
 for all homogeneous elements $x, y$ in $\mathcal{G}$ and $v\  \in\  V .$\\
Hence, we say that $(V,[.,.]_{V}, \beta)$ is a representation of $\mathcal{G}.$
\end{defn}
\begin{rem}
When $\beta$ is the zero-map, we say that the module $V$ is trivial.
\end{rem}

Let   $x_1,\cdots ,x_k$ be $k$ homogeneous elements of $\mathcal{G}$. We denote by $|(x_1,\cdots ,x_k)|=|x_1|+\cdots +|x_k| \  (\mod 2)$ the parity of an element $(x_{1},\dots,x_{k})$ in $\mathcal{G}^k$.\\
The set $C^{k}(\mathcal{G},V)$ of $k$-cochains on space $\mathcal{G}$ with values in $V,$ is the set of $k$-linear maps $f:\otimes^{k}\mathcal{G}\rightarrow V$
satisfying
$$f(x_{1},\dots,x_{i},x_{i+1},\dots,x_{k})=-(-1)^{|x_{i}||x_{i+1}|}f(x_{1},\dots,x_{i+1},x_{i},\dots,x_{k})\ \textrm{ for } 1\leq i\leq k-1 .$$
For $k=0$ we have $C^{0}(\mathcal{G},V)=V.$\\
The map $f$ is called even (resp. odd) when we have $f(x_{1},\dots,x_{k})\in V_0$ (resp. $f(x_{1},\dots,x_{k})\in V_1$)  for all even (resp odd ) element $(x_{1},\dots,x_{k})\in \mathcal{G}^k.$\\
A $k$-cochain on $\mathcal{G}$ with values in  $V$ is defined to be a $k$-hom-cochain $f \in C^{k}(\mathcal{G},\ V)$ such that it is compatible with $\alpha $ and $\beta$ in the sense that $\beta \circ f=f \circ \alpha,$ i.e.
$\beta \circ f (x_{1},\dots ,x_{k})=f(\alpha(x_{1}),\dots ,\alpha(x_{k})).$
Denote $C^{k}_{\alpha ,\beta}(\mathcal{G},\ V)$ the set of k-hom-cochains:
\begin{eqnarray}
C^{k}_{\alpha ,\beta}(\mathcal{G},\ V)&=&
\{f\in C^{k}(\mathcal{G},\ V) :\  \beta  \circ f=f \circ \alpha \}. \label{cohd}
\end{eqnarray}
Define $\delta^{k}:C^{k}(\mathcal{G},\ V)\rightarrow C^{k+1}(\mathcal{G},\ V)$ by setting

\begin{eqnarray*}
&&\delta^{k}(f)(x_{0},\dots,x_{k})\nonumber \\&&=\sum_{0\leq s < t\leq k}(-1)^{t+|x_{t}|(|x_{s+1}|+\dots+|x_{t-1}|)}
f\Big(\alpha(x_{0}),\dots,\alpha(x_{s-1}),[x_{s},x_{t}],\alpha(x_{s+1}),\dots,\widehat{x_{t}},\dots,\alpha(x_{k})\Big) \nonumber \\
&&+\sum_{s=0}^{k}(-1)^{s+|x_{s}|(|f|+|x_{0}|+\dots+|x_{s-1}|)}\Bigg[\alpha^{k-1}(x_{s}), f\Big(x_{0},\dots,\widehat{x_{s}},\dots,x_{k}\Big)\Bigg]_{V},
\end{eqnarray*}
 where $f\in C^{k}(\mathcal{G},\ V)$,  $|f|$ is the parity of $f$, $\ x_{0},....,x_{k}\in \mathcal{G}$ and $\ \widehat{x_{i}}\ $  means that $x_{i}$ is omitted.\\
\begin{thm}\cite{AMSa1}
Let $(\mathcal{G},[.,.],\alpha)$ be a multiplicative Hom-Lie superalgebra and $(V,[.,.]_V,\beta)$ be a $\mathcal{G}$-Hom-module.\\
The pair $(\oplus_{k>0}C^{k}_{\alpha ,\beta}(\mathcal{G},\ V),\{\delta^{k}\}_{k>0})$ defines a cohomology complex, that is
 $\delta^{k} \circ \delta^{k-1}=0.$
\end{thm}

Let $(\mathcal{G},[.,.],\alpha)$ be a multiplicative Hom-Lie superalgebra and $(V,[.,.]_V,\beta)$ be a $\mathcal{G}$-Hom-module.\\
We have with respect the cohomology defined by the coboundary operators
\begin{equation*}
 \delta^{k}:C^{k}_{\alpha ,\beta}(\mathcal{G},V)\longrightarrow    C^{k+1}_{\alpha ,\beta}(\mathcal{G},V).
\end{equation*}
\begin{itemize}
\item The $k$-cocycles space is defined as $Z^{k}(\mathcal{G},V)=\ker \ \delta^{k}.$ \\ The even (resp. odd ) k-cocycles  space is defined as
$Z^{k}_{0}(\mathcal{G},V)=Z^{k}(\mathcal{G},V)\cap (C^{k}_{\alpha ,\beta}(\mathcal{G},V))_{0}$ (resp. $Z^{k}_{1}(\mathcal{G},V)=Z^{k}(\mathcal{G},V)\cap (C^{k}_{\alpha ,\beta}(\mathcal{G},V))_{1}$.
 \item  The $k$-coboundary space is defined as    $B^{k}(\mathcal{G},V)=Im \ \delta^{k-1}.$ \\The  even (resp. odd ) k-coboundaries space is  $B^{k}_{0}(\mathcal{G},V)= B^{k}(\mathcal{G},V)\cap (C^{k}_{\alpha ,\beta}(\mathcal{G},V))_{0}$ (resp.  $B^{k}_{1}(\mathcal{G},V)= B^{k}(\mathcal{G},V)\cap (C^{k}_{\alpha ,\beta}(\mathcal{G},V))_{1}.$
  \item The $k^{th}$ cohomology  space is the quotient  $H^{k}(\mathcal{G},V)=Z^{k}(\mathcal{G},V)/ B^{k}(\mathcal{G},V). $ It decomposes as well as even and odd $k^{th}$ cohomology spaces.\\
 \end{itemize}
 Finally, we denote by $H^{k}(\mathcal{G},V)=H_{0}^{k}(\mathcal{G},V) \oplus H_{1}^{k}(\mathcal{G},V)$ the set $k^{th}$ cohomology space and by
 $\oplus_{k\geq 0}H^{k}(\mathcal{G},V)$ the cohomology group of the Hom-Lie superalgebra $\mathcal{G}$ with values in $V$.\\
  
  In the general case, let  $(\mathcal{G},[.,.],\alpha)$ be a Hom-Lie superalgebra. We have a $1$-coboundary (resp. $2$-coboundary operator) defined on  $\mathcal{G}$-valued cochains $C^{k}(\mathcal{G},\mathcal{G})$ such as  for $x,y,z\in\mathcal{G}$
  \begin{eqnarray}\label{cobound1}
 \delta^{1}(f)(x,y)=-f([x,y])+(-1)^{|x||f|} [x,f(y)]-(-1)^{|y|(|f|+|x|} [y,f(x)],
\end{eqnarray}
\begin{eqnarray}
\delta^{2}(f)(x,y,z)&=&-f([x,y],\alpha(z))+(-1)^{|z||y|}f([x,z],\alpha(y))+f(\alpha(x),[y,z])
\nonumber\\
&&+(-1)^{|x||f|}[\alpha(x),f(y,z)] -(-1)^{|y|(|f|+|x|)} [\alpha(y),f(x,z)]\nonumber\\
&&+(-1)^{|z|(|f|+|x|+|y|)}[\alpha(z),f(x,y)]. \ \ \ \, \label{cobound2}
\end{eqnarray}
A straightforward calculation shows that $\delta^{2}\circ\delta^{2}=0.$ We denote by $H^1(\mathcal{G},\mathcal{G})$ (resp. $H^2(\mathcal{G},\mathcal{G})$) the corresponding $1^{st}$ and $2^{nd}$ cohomology groups.
 \subsection{ Deformations of Hom-Lie superalgebras.}
 In this section we extend to Hom-Lie superalgebras the one-parameter formal deformation theory introduced by Gerstenhaber \cite{Gerst def} for associative algebras. It was extended to Hom-Lie algebras in \cite{MS3,AEM}.
 \begin{defn}
Let $( \mathcal{G},[.,.],\alpha)$ be a  Hom-Lie superalgebra. A one -parameter formal   deformation of $\mathcal{G}$ is given by the $\mathbb{K}[[t]]$-bilinear map $[.,.]_{t}:\mathcal{G}[[t]]\times\mathcal{G}[[t]]\longrightarrow\mathcal{G}[[t]]$ of the form
$\displaystyle [.,.]_{t}=\sum_{i\geq 0} t^i[.,.]_{i} $, 
where each $[.,.]_{i}$ is an even bilinear  map $[.,.]_{i}:\mathcal{G}\times\mathcal{G}\longrightarrow\mathcal{G}$ (extended to be $\mathbb{K}[[t]]$-bilinear) and 
$[.,.]=[.,.]_{0}$  satisfying the following conditions
\begin{eqnarray}
&[x,y]_{t}=-(-1)^{|x||y|}[y,x]_{t},\\ 
&(-1)^{|x||z|}[\alpha(x),[y,z]_{t}]_{t}+(-1)^{|z||y|} [\alpha(z),[x,y]_{t}]_{t}+(-1)^{|y||x|} [\alpha(y),[z,x]_{t}]_{t}=0. \label{djacobie}
\end{eqnarray}
The  deformation is said  of \emph{order} $k$ if $\displaystyle [.,.]_{t}=\sum_{i= 0}^{k} t^i[.,.]_{i}$.\\
Given two deformations $\mathcal{G}_{t}=( \mathcal{G},[.,.]_{t},\alpha)$ and  $(\mathcal{G}_{t}'= \mathcal{G},[.,.]'_{t},\alpha)$ of $\mathcal{G}$ where $\displaystyle [.,.]_{t}=\sum_{i\geq0}t^i[.,.]_{i}$ and
$\displaystyle [.,.]_{t}'=\sum_{i\geq0}t^i[.,.]_{i}'$ with $[.,.]_{0}=[.,.]_{0}'=[.,.].$  We say that  $\mathcal{G}_{t}$ and $\mathcal{G}_{t}'$
 are \emph{equivalent} if there exists a formal automorphism $\displaystyle \phi_{t}=\sum_{i\geq0} \phi_{i}t^i$ where $\phi_{i}\in End (\mathcal{G})$
 and $\phi_{0}=id_{\mathcal{G}}$, such that $$ \phi_{t}([x,y]_{t})=[\phi_{t}(x),\phi_{t}(y)]_{t}'.  $$
 A deformation $\mathcal{G}_{t}$ is said to be \emph{trivial} if and only if  $\mathcal{G}_{t}$ is equivalent to $\mathcal{G}$ (viewed as a superalgebra on $\mathcal{G}[[t]].$)
\end{defn}


The identity (\ref{djacobie}) is called deformation equation  and it is equivalent to
\begin{equation*}
    \circlearrowleft_{x,y,z} \sum_{i\geq0,j\geq0}(-1)^{|x||z|}t^{i+j}[\alpha(x),[y,z]_{i}]_{j}=0,
\end{equation*}
i.e.
\begin{equation*}
    \circlearrowleft_{x,y,z} \sum_{i\geq0,s\geq0}(-1)^{|x||z|}t^{s}[\alpha(x),[y,z]_{i}]_{s-i}=0,
\end{equation*}
or
\begin{equation*}
\sum_{s\geq0}  t^s   \circlearrowleft_{x,y,z}\sum_{i\geq0}(-1)^{|x||z|}t^{s}[\alpha(x),[y,z]_{i}]_{s-i}=0. 
\end{equation*}
The deformation equation  is equivalent to the following infinite system
\begin{equation}
  \circlearrowleft_{x,y,z}\sum_{i=0}^{s}(-1)^{|x||z|}[\alpha(x),[y,z]_{i}]_{s-i}=0,\ \textrm{for}\ s=0,1,2,\cdots \label{defordelta}
\end{equation}
In particular, For $s=0$ we have $\circlearrowleft_{x,y,z}(-1)^{|x||z|}[\alpha(x),[y,z]_{0}]_{0}=0$ which is the super Hom-Jacobi identity of $\mathcal{G}$.\\
The equation for $s=1$, is equivalent  to $\delta^{2}([.,.]_1)=0.$ Then $[.,.]_1$ is a $2$-cocycle ($[.,.]_1\in Z^{2}(\mathcal{G},\mathcal{G}$). We deal here with $\mathcal{G}$-valued cohomology.
For $s\geq2$, the identities (\ref{defordelta}) are equivalent to:
\begin{equation*}
   \delta^{2}([.,.]_{s}) (x,y,z)=-\circlearrowleft_{x,y,z}\sum_{i=1}^{s-1}(-1)^{|x||z|}[\alpha(x),[y,z]_{i}]_{s-i}.
\end{equation*}

Let $( \mathcal{G},[.,.],\alpha)$ be a  Hom-Lie superalgebra and $[.,.]_{1}$ be an element of $Z^{2}(\mathcal{G},\mathcal{G})$. The $2$-cocycle
$[.,.]_{1}$ is said to be integrable if there exists a family $([.,.]_{p})_{p\geq0}$ such that $\displaystyle [.,.]_{t}=\sum_{i\geq0}t^i[.,.]_{i}$ defines a formal deformation $\mathcal{G}_{t}=( \mathcal{G},[.,.]_{t},\alpha)$  of $ \mathcal{G}$.\\

One may also prove:
\begin{thm}
Let $( \mathcal{G},[.,.],\alpha)$ be a  Hom-Lie superalgebra and  $\mathcal{G}_{t}=( \mathcal{G},[.,.]_{t},\alpha)$  be a one-parameter formal deformation of $\mathcal{G}$, where $\displaystyle [.,.]_{t}=\sum_{i\geq0}t^i[.,.]_{i}$. Then there exists an equivalent deformation  $(\mathcal{G}_{t}'= \mathcal{G},[.,.]'_{t},\alpha)$, where
$\displaystyle [.,.]_{t}'=\sum_{i\geq0}t^i[.,.]_{i}'$ such that $[.,.]_{1}'\in Z^{2}(\mathcal{G},\mathcal{G})$ and doesn't belong to $B^{2}(\mathcal{G},\mathcal{G})$.\\

Hence, if $H^2(\mathcal{G},\mathcal{G})=0$ then every formal deformation is equivalent to a trivial deformation. The Hom-Lie superalgebra is called rigid.
\end{thm}
 
\section{  The  $q$-Witt superalgebras}
Let $\mathcal{A}=\mathcal{A}_{0}\oplus\mathcal{A}_{1}$ be an associative superalgebra. We assume that $\mathcal{A}$ is super-commutative, that is for homogeneous elements $a, b$, the identity $ab=(-1)^{|a||b|}ba$ holds.  

\begin{defn}
A $\sigma$-derivation $D_{i}$ ($i\in \mathbb{Z}_{2}$) on $ \mathcal{A}$ is an endomorphism satisfying:
$$D_{i}(ab)=D_{i}(a)b+(-1)^{i|a|}\sigma (a)D_{i}(b),$$
where $ a, b\ \in \mathcal{A}$ are homogeneous element and $|a|$ is the parity of $a$.\\
A $\sigma$-derivation $D_0$ is said to be an  even $\sigma$-derivation  and $D_1$ is an  odd $\sigma$-derivation. The set of
all $\sigma$-derivations is denoted by $Der_{\sigma}( \mathcal{A})$. Therefore, $Der_{\sigma}(\mathcal{A})=Der_{\sigma}(\mathcal{A})_{0}\oplus Der_{\sigma}(\mathcal{A})_{1}$, where
$Der_{\sigma}(\mathcal{A})_{0}$ (resp $Der_{\sigma}(\mathcal{A})_{1}$) is the space of even (resp. odd) $\sigma$-derivations.
\end{defn}

Let $\mathcal{A}=\mathcal{A}_{0}\oplus\mathcal{A}_{1}$ be a super-commutative  associative superalgebra,  such $\mathcal{A}_{0}=\mathbb{C}[t,t^{-1}]$ and $\mathcal{A}_{1}=\theta\mathcal{A}_{0}$ where $\theta$ is the Grassman variable $(\theta^{2}=0).$
We set $\{n\}=\frac{1-q^{n}}{1-q}, $ a $q$-number, where $q\in \mathbb{C}\backslash \{0,1\}$ and $n\in \mathbb{N}$.
Let $\sigma$ be the algebra endomorphism on $\mathcal{A}$ defined by  $$\sigma(t^{n})=q^{n}t^{n} \ \ and \ \ \sigma (\theta)=q\theta.$$
 
Let $\partial_{t}$ \ and \ $\partial_{\theta}$\ be two linear maps on $\mathcal{A}$ defined by $$\partial_{t}(t^{n})=\{n\}t^{n}, \ \partial_{t}(\theta t^{n})=\{n\}\theta t^{n},$$ $$\partial_{\theta}(t^{n})=0, \ \ \partial_{\theta}(\theta t^{n})=q^{n}t^{n}.$$

\begin{lem}\cite{AmmarMakhloufJA2010}
The linear map $\Delta=\partial_{t}+\theta \partial_{\theta}$ on  $ \mathcal{A}$ is an even $\sigma$-derivation.\\Hence, $\Delta(t^{n})=\{n\}t^{n}$ and 
$\Delta(\theta t^{n})=\{n+1\}\theta t^{n}.$
\end{lem}

Let $\mathcal{W}^q= \mathcal{A}\cdot\Delta$ be a superspace generated by  elements $L_{n}=t^{n}\cdot \Delta$ of parity $0$ and  elements $G_{n}=\theta t^{n}\cdot\Delta$ of parity 1.\\
Let $[-,-]$ be a bracket on the superspace $\mathcal{W}^{q}$ defined by
\begin{eqnarray}
&&[L_{n},L_{m}]=(\{m\}-\{n\})L_{n+m}, \label{crochet1} \\
&&[L_{n},G_{m}]=(\{m+1\}-\{n\})G_{n+m}.\label{crochet2}
 \end{eqnarray}
 The others brackets are obtained by supersymmetry or equals $0.$\\
It is easy to see that $\mathcal{W}^{q}$ is a $\mathbb{Z}-$graded algebra $
 \mathcal{W}^{q}=\oplus_{n\in\mathbb{Z}}\mathcal{W}^{q}_n,
$ 
 where
$
 \mathcal{W}^{q}_{n}=span_{\mathbb{C}} \{L_{n}, G_{n}\}
$. The elements $L_n$ and $G_n$ are said of degree $n$.

Let $\alpha$ be an even linear map on $\mathcal{W}^{q}$ defined on the generators by
 \begin{equation}\label{alpha1}
    \alpha(L_{n})=(1+q^{n})L_{n},
 \end{equation}
 \begin{equation}\label{alpha2}
   \alpha(G_{n})=(1+q^{n+1})G_{n}.
 \end{equation}
\begin{prop}\cite{AmmarMakhloufJA2010}
The triple $(\mathcal{W}^{q},[-,-],\alpha)$ is a Hom-Lie superalgebra.
\end{prop}
In the sequel we refer to this Hom-Lie superalgebras as $\mathcal{W}^{q}$. We call it $q$-deformed Witt superalgebra.
\section{Second cohomology $H^{2}(\mathcal{W}^{q},\mathcal{W}^{q})$}
In this section, we aim to compute the second cohomology group  of $\mathcal{W}^{q}$ with values in itself. 
For all  $q$-deformed 1-cocycle (resp 2-cocycle) on $\mathcal{W}^q$ we have  with respect to \eqref{cobound1},\eqref{cobound2},  respectively
\begin{eqnarray}\label{ddd1}
 0&=& \delta^{1}(f)(x_{0},x_{1})\nonumber\\
 &=&-f([x_{0},x_{1}])+(-1)^{|x_{0}||f|} [x_{0},f(x_{1})]-(-1)^{|x_{1}|(|f|+|x_{0}|} [x_{1},f(x_{0})],
\end{eqnarray}
\begin{eqnarray}
0&=&\delta^{2}(f)(x_{0},x_{1},x_{2})\nonumber\\
&=&-f([x_{0},x_{1}],\alpha(x_{2}))+(-1)^{|x_{2}||x_{1}|}f([x_{0},x_{2}],\alpha(x_{1}))+f(\alpha(x_{0}),[x_{1},x_{2}])
\nonumber\\
&&+(-1)^{|x_{0}||f|}[\alpha(x_{0}),f(x_{1},x_{2})] -(-1)^{|x_{1}|(|f|+|x_{0}|)} [\alpha(x_{1}),f(x_{0},x_{2})]\nonumber\\
&&+(-1)^{|x_{2}|(|f|+|x_{0}|+|x_{1}|)}[\alpha(x_{2}),f(x_{0},x_{1})]. \ \ \ \, \label{ddd}
\end{eqnarray}

Taking the pair $(x, y)$ to be respectively   $(L_{n}, L_{p})$ and $ (L_{n},G_{p})$
in (\ref{ddd1}), we obtain
\begin{eqnarray}
\begin{cases}
 \delta^{1}(f)(L_{n},L_{p})=-f([L_{n}, L_{p}])
+[L_{n},f(L_{p})]-[L_{p},f(L_{n}]=0.\label{1cochain}
\\
\delta^{1}(f)(L_{n},G_{p})=-f([L_{n}, G_{p}])
+[L_{n},f(G_{p})]-(-1)^{|f|}[G_{p},f(L_{n})]=0.\label{1cochain2}
\end{cases}
\end{eqnarray}

Taking the triple $(x, y, z)$ to be $(L_{n}, L_{m}, L_{p}),\ (L_{n}, L_{m},G_{p}),\  and \ (L_{n},G_{m} ,G_{p})$
in (\ref{ddd}), respectively, we obtain
\begin{eqnarray}
 0&=&-f([L_{n}, L_{m}],\alpha(L_{p}))+f([L_{n}, L_{p}],\alpha(L_{m}))+f(\alpha(L_{n}),[L_{m}, L_{p}])\nonumber \\
&&+[\alpha(L_{n}),f(L_{m},L_{p})]-[\alpha(L_{m}),f(L_{n},L_{p})]+[\alpha(L_{p}),f(L_{n},L_{m})].\label{pair}
\\
 0&=&-f([L_{n}, L_{m}],\alpha(G_{p}))+f([L_{n}, G_{p}],\alpha(L_{m}))+f(\alpha(L_{n}),[L_{m}, G_{p}])\nonumber\\
&&+[\alpha(L_{n}),f(L_{m},G_{p})]-[\alpha(L_{m}),f(L_{n},G_{p})]+(-1)^{|f|}[\alpha(G_{p}),f(L_{n},L_{m})].\label{pair2}
\\
0&=&-f([L_{n}, G_{m}],\alpha(G_{p}))-f([L_{n}, G_{p}],\alpha(G_{m}))+f(\alpha(L_{n}),[G_{m}, G_{p}])\nonumber\\
&&+[\alpha(L_{n}),f(G_{m},G_{p})]-(-1)^{|f|}[\alpha(G_{m}),f(L_{n},G_{p})]-(-1)^{|f|}[\alpha(G_{p}),f(L_{n},G_{m})].\label{pair3}
\end{eqnarray}

Our main theorem is 
 
\begin{thm}
The second cohomology group of  $q$-deformed Witt superalgebras $W^{q}$ with values in the adjoint module vanishes, i.e.
$$ H^{2}(\mathcal{W}^{q},\mathcal{W}^{q})=\{0\}.$$
Hence, every formal deformation is equivalent to a trivial deformation. 
\end{thm}

In the sequel we proceed to prove this result by computing the second even and odd  cohomology groups. We have 
$$H^{2}(\mathcal{W}^{q},\mathcal{W}^{q})=H_0^{2}(\mathcal{W}^{q},\mathcal{W}^{q})\oplus H_1^{2}(\mathcal{W}^{q},\mathcal{W}^{q})$$
where $H_0^{2}(\mathcal{W}^{q},\mathcal{W}^{q})$ (resp. $H_1^{2}(\mathcal{W}^{q},\mathcal{W}^{q})$) is the  even (resp. odd) subspace.

\subsection{Second even cohomology $H^{2}_{0}(\mathcal{W}^{q},\mathcal{W}^{q})$}
We denote by $H^{2}_{0,s}(\mathcal{W}^{q},\mathcal{W}^{q})$ the even subspace  of degree $s$. That is given by even  $2$-cochains    of degree $s$, i.e.    for all homogeneous elements $x_1,x_2\in\mathcal{W}^{q}$ of degree respectively $m,n$, $f(x_1, x_2)$ is  of degree $m+n+s$.\\

 Assume now that $f$ is an even $2$-cocycle of degree $s$. We  set
\begin{equation*}
   f(L_n,L_p)=a_{s,n,p}L_{s+n+p},\  f(L_n,G_p)=b_{s,n,p}G_{s+n+p} \textrm{ and} f(G_n,G_p)=c_{s,n,p}L_{s+n+p}.
\end{equation*}
When there is no ambiguity with the degree $s$, the coefficients 
$a_{s,n,p},\  b_{s,n,p},\ c_{s,n,p}$ are denoted by $a_{n,p},\  b_{n,p},\ c_{n,p}$.

By (\ref{pair}), we have
\begin{eqnarray}\label{pairds}
0&=&-(1+q^{p})(\{m\}-\{n\})a_{n+m,p}+(1+q^{m})(\{p\}-\{n\})a_{n+p,m}\nonumber \\
&&+(1+q^{n})(\{p\}-\{m\})a_{n,m+p}
+(1+q^{n})(\{m+p+s\}-\{n\}) a_{m,p} \\
&& - (1+q^{m})(\{p+n+s\}-\{m\})a_{n,p}
+ (1+q^{p})(\{n+m+s\}-\{p\})a_{n,m}.\nonumber
\end{eqnarray}
Therefore by  (\ref{pair2}), we obtain
\begin{eqnarray}\label{pairdsd}
0&=&-(1+q^{p+1})(\{m\}-\{n\})b_{n+m,p}-(1+q^{m})(\{p+1\}-\{n\})b_{m,n+p}\nonumber \\
&&+(1+q^{n})(\{p+1\}-\{m\})b_{n,m+p}
+(1+q^{n})(\{m+p+1+s\}-\{n\}) b_{m,p} \\
&& - (1+q^{m})(\{p+n+1+s\}-\{m\})b_{n,p}
- (1+q^{p+1})(\{p+1\}-\{n+m+s\})a_{n,m}.\nonumber
\end{eqnarray}

Since $[G_{n},G_{m}]=0$, by (\ref{pair3}), we obtain
\begin{eqnarray}
0&=&-(1+q^{p+1})(\{m+1\}-\{n\})c_{n+m,p}-(1+q^{m+1})(\{p+1\}-\{n\})c_{n+p,m}  \nonumber \\
&&+(1+q^{n})(\{m+p+s\}-\{n\}) c_{m,p}.
 \label{paird}
\end{eqnarray}

\begin{prop}\label{lemma}
If $s\neq 0,\ 2$, the subspaces $H^{2}_{0,s}(\mathcal{W}^{q},\mathcal{W}^{q})$ 
 vanishe.
\end{prop}
\begin{proof}
We define an endomorphism $g$ of $\mathcal{W}^{q}$ by
\begin{equation*}
   g(L_{p})=\frac{1}{q^{p}\{s\}}f(L_{0},L_{p})  \textrm{ and} \ g(G_{p})=\frac{1}{q^{p+1}\{s\}}f(L_{0},G_{p}).
\end{equation*}
By (\ref{1cochain}) and (\ref{crochet1}) we have
\begin{equation*}
\delta^{1}(g)(L_{0},L_{p}) =-\{p\}g(L_{p})+\{p+s\}g(L_{p}).
\end{equation*}
So
\begin{equation*}
\delta^{1}(g)(L_{0},L_{p})=q^{p}\{s\}g(L_{p}).
\end{equation*}
We define a $2$-cocycle $h$ by
\begin{equation*}
h=f-\delta ^{1} (g).
\end{equation*}
Therefore
 \begin{equation}
h(L_{0},L_{p})=f(L_{0},L_{p})-\delta^{1}(g)(L_{0},L_{p})
=q^{p}\{s\}g(L_{p})-q^{p}\{s\}g(L_{p})=0.\label{zero}
\end{equation}
 Taking $m=0$ in (\ref{pair}), with (\ref{crochet1}) and (\ref{alpha1})    we obtain
\begin{eqnarray}
0&=&(1+q^{p+1})\{n\}f(L_{n},L_{p})+2(\{p\}-\{n\})f(L_{n+p},L_{0})+(1+q^{n})\{p\}f(L_{n}, L_{p})\nonumber\\
&&+(1+q^{n})[L_{n},f(L_{0},L_{p})]-2[L_{0},f(L_{n},L_{p})]+(1+q^{p})[L_{p},f(L_{n},L_{0})] \label{mars}.
\end{eqnarray}
Since $h$ is a $2$-cocycle,  we can replace $f$ by $h$ in (\ref{mars}), and using (\ref{zero}) we obtain
\begin{eqnarray*}
0&=&(1+q^{p})\{n\}h(L_{n},L_{p})+(1+q^{n})\{p\}h(L_{n}, L_{p})-2\{s+n+p\}h(L_{n},L_{p}).
\end{eqnarray*}
From this, using  the fact that $s$ is not vanishing, we obtain
\begin{equation}\label{lemma1}
  h(L_{n},L_{p})=0\ \forall \ n,p\in \mathbb{Z}.
\end{equation}
 Since  $g(G_{p})=\frac{1}{q^{p+1}\{s\}}f(L_{0},G_{p})$,  by (\ref{1cochain2}) and (\ref{crochet2}) we have
\begin{equation*}\label{}
  \delta^{1}(g)(L_{0},G_{p})=-\{p+1\}g(G_{p})+\{p+s+1\}g(G_{p})=q^{p+1}\{s\}g(G_{p}).
\end{equation*}
Then
\begin{equation*}
h(L_{0},G_{p})=f(L_{0},G_{p})-\delta^{1}(g)(L_{0},G_{p})
=0.
\end{equation*}
Taking $m=0$ in (\ref{pair2}),\ by (\ref{crochet1}), (\ref{crochet2}), (\ref{alpha1}) and (\ref{alpha2}),  we obtain
\begin{eqnarray*}
&&(1+q^{p})\{n\}f(L_{n},G_{p})+2(\{p+1\}-\{n\})f( G_{p+n},L_{0})+(1+q^{n})\{p+1\}f(L_{n}, G_{p})\\
&&+(1+q^{n})[L_{n},f(L_{0},G_{p})]-2[L_{0},f(L_{n},G_{p})]+(-1)^{|f|}(1+q^{p+1})[G_{p},f(L_{n},L_{0})]=0.
\end{eqnarray*}
Since $h$ is a $2$-cocycle and $h(L_{0},G_{p})=h(L_{0},L_{n})=0$ we can deduce
\begin{eqnarray*}
(1+q^{p+1})\{n\}h(L_{n},G_{p})+(1+q^{n})\{p+1\}h(L_{n},G_{p})-2\{p+n+s+1\}h(L_{n},G_{p})=0,
\end{eqnarray*}
which implies, under the condition $s\neq0 $,  that
\begin{eqnarray}\label{hlemma2}
 &&\ h(L_{n},G_{p})=0.
\end{eqnarray}
Taking $n=0$ in (\ref{pair3}), since $[G_{m},G_{p}]=h(L_{n},G_{p})=h(L_{n},L_{p})=0$ and $h$ is a $2$-cocycle we have
\begin{eqnarray*}
-h([L_{0},G_{m}],\alpha(G_{p}))-h([L_{0},G_{p}],\alpha(G_{m}))+[\alpha(L_{0}),h(G_{m},G_{p})]=0.
\end{eqnarray*}
Then
\begin{equation*}
- (1+q^{p+1})\{m+1\} h( G_{m},G_{p}) - (1+q^{m+1})\{p+1\}h(G_{p},G_{m})+2\{m+p+s\}h(G_{m},G_{p})=0,
\end{equation*}
which implies,      under the condition $s\neq2$,  that
$h( G_{m},G_{p})=0.$\\
It follows that $h\equiv 0$. Hence $f$ is a coboundary.
\end{proof}

\begin{lem}
Let $f$  be an  even $2$-cocycle of degree zero $(s=0)$ and  $g$ be an even  endomorphism  of $\mathcal{W}^{q}$. If $h=f-\delta^{1}(g)$
then $$h(L_{-1},L_2)=0,\ h(L_{-1},G_{1})=0,\ h(L_{n},L_1)=0\  \textrm{and} \   h(L_1,G_m)=0\ \forall n\in \mathbb{Z},\ m\in \mathbb{Z}^{*}.$$
\end{lem}
\begin{proof}
  Let $f$ be an  even $2$-cocycle. Then $f(L_{n},L_{m})=f_{n,m}L_{n+m}$ and  $f(L_{n},G_{m})=f_{n,m}^{'}G_{n+m}$ .

Let  $(a_{n})_{n\in\mathbb{Z}}$  be  the sequence  given recursively by
  \begin{eqnarray*}
    &&a_{0}=f_{0,1},\\
&&a_{n}=a_{n+1}+\frac{1}{\{1\}-\{n\}} f_{n,1}  \ \ \forall n<0,\\
     &&a_1=0,\\
   &&a_{2} =\frac{1}{\{2\}-\{-1\}}f_{-1,2}-a_{-1},
\\
&&a_{n+1}=a_{n}+\frac{1}{\{n\}-\{1\}} f_{n,1}  \ \ \forall n\geq2.
   \end{eqnarray*}
 Let  $(b_{m})_{m\in\mathbb{Z}}$  be  the sequence  given recursively by
    \begin{eqnarray*}
&&b_0=0,\\
&&b_{m}=b_{m+1}- \frac{1}{\{1\}-\{m+1\}}f_{1,m}^{'}  \ \ \forall m<0,\\
    && b_{1}=-a_{-1}-\frac{1}{\{-1\}-\{2\}}f_{-1,1}^{'},
\\
&&b_{m+1}=b_{m}+\frac{1}{\{1\}-\{m+1\}} f_{1,m}'  \ \ \forall m\geq1.
   \end{eqnarray*}
Let $g$ be an even  endomorphism  of $\mathcal{W}^{q}$ given  by
$g(L_{n})=a_{n}L_{n}$ and $g(G_{n})=b_{n}G_{n}.$\\
By (\ref{1cochain}) and (\ref{1cochain2}) we have  recursively
 \begin{eqnarray*}
\delta^{1}(g)(L_{n},L_{m})&=&(\{n\}-\{m\})(a_{n+m}-a_{m}-a_{n})L_{n+m},
  \end{eqnarray*}
  and
   \begin{eqnarray*}
\delta^{1}(g)(L_{n},G_{m})&=&(\{n\}-\{m+1\})(b_{n+m}-b_{m}-a_{n})G_{n+m}.
  \end{eqnarray*}
   If $h=f-\delta^{1}(g)$ we have
 \begin{eqnarray*}
h_{n,1}&=&f_{n,1}-(\{n\}-\{1\})(a_{n+1}-a_{n})=0\ \forall n\in \mathbb{Z},\\
h_{-1,2}&=&f_{-1,2}-(\{-1\}-\{2\})(a_{1}-a_{2}-a_{-1})=0,\\
h_{1,m}^{'}&=&f_{1,m}^{'}-(\{1\}-\{m+1\})(b_{m+1}-b_{m})=0\ \forall m\in \mathbb{Z}^{*},\\
h_{-1,1}^{'}&=& f_{-1,1}^{'}+(\{-1\}-\{2\})(b_{1}+a_{-1})=0.
  \end{eqnarray*}
     This proves the lemma.
\end{proof}

\begin{lem}\label{lemma3.3}
Let $f$ be an even $2$-cocycle of degree zero such that $f(L_{n},L_{1})=0$ and $f(L_{-1},L_{2})=0$. 
Then the cohomology class of  $f$ is trivial  on the space $\mathcal{W}_{0}^{q}\times \mathcal{W}_{0}^{q}$.
\end{lem}
\begin{proof}
Since $f$ is an even
$2$-cocycle  of degree zero, by (\ref{pairds}) we have
\begin{eqnarray}
&&-(1+q^{p})(\{m\}-\{n\})a_{n+m,p}+ (1+q^{m})(\{p\}-\{n\})a_{n+p,m}+(1+q^{n})(\{p\}-\{m\})a_{n,m+p} \nonumber\\
&&+(1+q^{n})(\{m+p\}-\{n\})a_{m,p}-(1+q^{m})(\{n+p\}-\{m\})a_{n,p}+(1+q^{p})(\{n+m\}-\{p\})a_{n,m} \nonumber\\
&&=0.  \label{formul 1}\ \ \ \ \ \ \ \ \
\end{eqnarray}
Taking $m=-1$ in (\ref{formul 1}), we have
\begin{eqnarray}
&&-(1+q^{p})(\{-1\}-\{n\})a_{n-1,p}+ (1+q^{-1})(\{p\}-\{n\})a_{n+p,-1}+(1+q^{n})(\{p\}-\{-1\})a_{n,-1+p}\nonumber \\
&&+(1+q^{n})(\{-1+p\}-\{n\})a_{-1,p}-(1+q^{-1})(\{n+p\}-\{-1\})a_{n,p}+(1+q^{p})(\{n-1\}-\{p\})a_{n,-1}\nonumber \\
&&=0.\label{coh1} \ \ \ \ \ \ \ \
\end{eqnarray}
Setting $m=1$ in (\ref{formul 1}), since $a_{1,k}=-a_{k,1}=0 \ \forall k\in \mathbb{Z}$, we have
\begin{eqnarray}
-(1+q^{p})(\{1\}-\{n\})a_{n+1,p}+(1+q^{n})(\{p\}-\{1\})a_{n,1+p}
-(1+q)(\{n+p\}-\{1\})a_{n,p}
=0. \label{cohd}  
\end{eqnarray}
We investigate the  following cases:\\
\textbf{ Case 1: $k=0$}\\
In (\ref{cohd}) we consider $p=0$ this gives
\begin{eqnarray}
-2(\{1\}-\{n\})a_{n+1,0}
-(1+q)(\{n\}-\{1\})a_{n,0}
=0.  
\end{eqnarray}
That is,
\begin{equation*}
 a_{n,0}=\frac{2}{1+q} a_{n+1,0},\ \   a_{n+1,0}=\frac{1+q}{2} a_{n,0}, \ \forall n\neq 1.
\end{equation*}

Starting from $a_{1,0}=-a_{0,1}=0$ this implies for $n\leq 0$ that $a_{n,0}=0$ and for $n\geq 3$ that $a_{n,0}=\big(\frac{1+q}{2}\big)^{n-2}a_{2,0}$.\\
Next we consider (\ref{coh1}) for $ n=0,\ p=2$. It follows
\begin{eqnarray*}
&&-(1+q^{2})\{-1\}a_{-1,2}+ (1+q^{-1})\{2\}a_{2,-1}+2(\{2\}-\{-1\})a_{0,1}\\
&&+2\{1\}a_{-1,2}-(1+q^{-1})(\{2\}-\{-1\})a_{0,2}+(1+q^{2})(\{-1\}-\{2\})a_{0,-1}\\
&&=0. \ \ \ \ \ \ \ \
\end{eqnarray*}
The term $a_{-1,2}$ cancels and we know already that $a_{0,1}=a_{0,-1}=0$.  Hence $a_{0,2}=0$.\\ This implies
\begin{equation*}
    a_{n,0}=0\ \ \forall n\in \mathbb{Z}.
\end{equation*}
\textbf{ Case 2: $k=-1$}\\
In (\ref{cohd}) we set $p=-1$ and obtain (with $a_{n,0}=0$)
\begin{eqnarray*}
-(1+q^{-1})(\{1\}-\{n\})a_{n+1,-1}
-(1+q)(\{n-1\}-\{1\})a_{n,-1}
=0.   
\end{eqnarray*}
Hence,
\begin{equation*}
    a_{n,-1}=-\frac{(1+q^{-1})(\{1\}-\{n\})}{(1+q)(\{n-1\}-\{1\})}a_{n+1,-1}\  \forall n\neq 2,\ \  a_{n+1,-1}=-\frac{(1+q)(\{n-1\}-\{1\})}{(1+q^{-1})(\{1\}-\{n\})}a_{n,-1} \ \forall n\neq 1.
\end{equation*}
The first formula, starting from $a_{1,-1}=-a_{-1,1}=0$, implies that $ a_{n,-1}=0,$ for all $n\leq 1$.\\
The second formula, for $n=2$, implies $ a_{3,-1}=0$ and hence $a_{n,-1}=0$ for $n\geq 3$. But by assumption $a_{2,-1}=-a_{-1,2}=0.$ Hence,
\begin{equation*}
    a_{n,-1}=0\ \ \forall n\in \mathbb{Z}.
\end{equation*}
\textbf{ Case 3: $k=-2$}\\
We plug the value $p=-2$ into (\ref{cohd}) and get for terms not yet identified as zero
\begin{eqnarray*}
-(1+q^{-2})(\{1\}-\{n\})a_{n+1,-2}
-(1+q)(\{n-2\}-\{1\})a_{n,-2}
=0.   
\end{eqnarray*}
This yields
\begin{equation*}
    a_{n+1,-2}=\frac{(1+q)(\{1\}-\{n-2\})}{(1+q^{-2})(\{1\}-\{n\})}a_{n,-2}\  \forall n\neq 1,\ \ a_{n,-2}=\frac{(1+q^{-2})(\{n\}-\{1\})}{(1+q)(\{n-2\}-\{1\})}a_{n+1,-2} \ \forall n\neq 3.
\end{equation*}
From the first formula we get $ a_{3,-2}=\frac{1+q}{(1+q^{-2})(\{1\}-\{2\})}a_{2,-2},\  a_{4,-2}=0$ and hence $ a_{n,-2}=0$ for all $n\geq 4$.\\
From the second formula we get starting from $a_{1,-2}=0$ that $ a_{n,-2}=0$ for all $n\leq 1$.\\
Finally, $ a_{n,-2}=0$ for  $n\neq 2,\ 3$. The value $ a_{3,-2}=\frac{1+q}{(1+q^{-2})(\{1\}-\{2\})}a_{2,-2}$ stays undetermined for the moment.\\
\textbf{ Case 4: $k=2$}\\
We start from (\ref{coh1}) for $p=2$ and recall that terms of levels $0,\ 1,\ -1$ are zero. This gives
\begin{eqnarray*}
-(1+q^{2})(\{-1\}-\{n\})a_{n-1,2}
-(1+q^{-1})(\{n+2\}-\{-1\})a_{n,2}
=0. 
\end{eqnarray*}
Hence,
\begin{equation*}
    a_{n,2}=\frac{(1+q^{2})(\{n\}-\{-1\})}{(1+q^{-1})(\{n+2\}-\{-1\})}a_{n-1,2}\  \forall n\neq -3,\ \  a_{n-1,2}=\frac{(1+q^{-1})(\{n+2\}-\{-1\})}{(1+q^{2})(\{n\}-\{-1\})}a_{n,2} \ \forall n\neq -1.
\end{equation*}
In  the first formula we start from $a_{-1,2}=0$ and get $a_{n,2}=0,\ \forall n\geq -1.$\\
From the second, we get $ a_{-3,2}=-\frac{(1+q^{-1})\{-1\}}{(1+q^{2})(\{-2\}-\{-1\})}a_{-2,2}$, then $a_{-4,2}=0$ and then altogether $a_{n,2}=0$ for all $n\neq \ -2,\ -3.$\\
The value  $ a_{-3,2}=-\frac{(1+q^{-1})\{-1\}}{(1+q^{2})(\{-2\}-\{-1\})}a_{-2,2}$ stays undetermined for the moment.
To find it we consider the index triple $(2,-2,4)$ in (\ref{formul 1}) and obtain after leaving out terms which are obviously zero
$
-(1+q^{4 })\{4\}a_{2,-2}
=0.$

This shows that $a_{3,-2}=a_{-3,2}=a_{2,-2}=0$ and we can conclude
$
    a_{n,-2}=a_{n,2}=0\ \ \forall n\in \mathbb{Z}.
$

\textbf{ Case 5: $k<-2$}\\
We make induction assuming it true for $k=2,\ 1,0\ ,-1,\ -2.$ We start from (\ref{formul 1}) for $p$ and put  $p=-1$. We get 
\begin{equation}
 (1+q^{m})(\{-1\}-\{n\})a_{n-1,m}+(1+q^{n})(\{-1\}-\{m\})a_{n,m-1}
+(1+q^{-1})(\{n+m\}-\{-1\})a_{n,m}
=0.    \label{coh2}
\end{equation}
Then
\begin{equation*}
 (1+q^{n})(\{-1\}-\{m\})a_{n,m-1}= (1+q^{-1})(\{-1\}-\{n+m\})a_{n,m}-(1+q^{m})(\{-1\}-\{n\})a_{n-1,m}.
\end{equation*}
We deduce that $a_{n,m}=0, \forall \ m<-2$.\\
\textbf{ Case 6: $k>2$}\\
We make induction assuming it true for $n=2.$ By (\ref{cohd}) we have
\begin{eqnarray*}
(1+q^{p})(\{1\}-\{n\})a_{n+1,p}&=&(1+q^{n})(\{p\}-\{1\})a_{n,1+p}-
(1+q)(\{n+p\}-\{1\})a_{n,p}.
\end{eqnarray*}
We deduce that $a_{n,m}=0, \forall \ n>2$.\\
\end{proof}

\begin{lem}\label{lemma3.4}
Let $f$ be an even $2$-cocycle of degree zero such that $f(L_{n},L_{1})=0$ , $f(L_{-1},L_{2})=0$, $f(L_{1},G_{m})=0, \forall m\in \mathbb{Z}^{*}$ and $f(L_{-1},G_{1})=0$. 
Then the cohomology class of  $f$ is trivial  on the space $\mathcal{W}_{0}^{q}\times\mathcal{W}_{1}^{q}$.
\end{lem}
\begin{proof}
Since $f$ is an even
$2$-cocycle  of degree zero and $f(L_{n},L_{m})=0$,
by (\ref{pairdsd}) we obtain
\begin{eqnarray}\label{cobbbss}
&&-(1+q^{p+1})(\{m\}-\{n\})b_{n+m,p}-(1+q^{m})(\{p+1\}-\{n\})b_{m,n+p}\nonumber \\
&&+(1+q^{n})(\{p+1\}-\{m\})b_{n,m+p}
+(1+q^{n})(\{m+p+1\}-\{n\}) b_{m,p} \\
&& - (1+q^{m})(\{p+n+1\}-\{m\})b_{n,p}=0
.\nonumber
\end{eqnarray}
In (\ref{cobbbss}) we set  $m=1$   and obtain
\begin{eqnarray}
&&-(1+q^{p+1})(\{1\}-\{n\})b_{n+1,p}-(1+q)(\{p+1\}-\{n\})b_{1,n+p}
+(1+q^{n})(\{p+1\}-\{1\})b_{n,1+p}
\nonumber \\ && +(1+q^{n})(\{p+2\}-\{n\}) b_{1,p}-(1+q)(\{n+p+1\}-\{1\})b_{n,p}
=0.\ \ \ \ \ \ \ \ \label{coh3}
\end{eqnarray}
Then (for $p\neq 0,\ n+p\neq0$)
\begin{eqnarray}
-(1+q^{p+1})(\{1\}-\{n\})b_{n+1,p}
+(1+q^{n})(\{p+1\}-\{1\})b_{n,1+p}
-(1+q)(\{n+p+1\}-\{1\})b_{n,p}
=0. \nonumber\\ \ \label{coh31}
\end{eqnarray}

We consider in the following the  different cases:\\
\textbf{ Case 1: $k=0$}\\
In (\ref{coh31}),  we set $n=0$ and obtain (for $p\neq 0$)
\begin{eqnarray}
2(\{p+1\}-\{1\})b_{0,1+p}
-(1+q)(\{p+1\}-\{1\})b_{0,p}
=0.\ \ \ \ \ \ \ \
\end{eqnarray}
Hence
\begin{equation*}
  b_{0,1+p}=\frac{1+q}{2} b_{0,p},\ for\  p\neq 0,\   
  \end{equation*}
which gives
  \begin{equation}\label{formula2}
 b_{0,p}=\big(\frac{1+q}{2} \big)^{p-1} b_{0,1}  ,\ for\  p> 0,\   .
  \end{equation}
Letting  $m=0,\ p=1$ in (\ref{cobbbss}), we get
\begin{eqnarray*}
&&(1+q^{2})\{n\}b_{n,1}- 2(\{2\}-\{n\})b_{0,n+1}
+(1+q^{n})\{2\}b_{n,1}
+(1+q^{n})(\{2\}-\{n\})b_{0,1}-2\{n+2\}b_{n,1}
=0.  \ \ \ \ \ \ \ \
\end{eqnarray*}
Hence
\begin{equation}\label{medje}
- 2(\{2\}-\{n\})b_{0,n+1}+(1+q^{n})(\{2\}-\{n\})b_{0,1}
=0.
\end{equation}
Setting $n=3$ in (\ref{medje}), we obtain
 $b_{0,4}  = \frac{1+q^3}{2} b_{0,1} $. Furthermore, taking $p=4$ in (\ref{formula2}), one has $b_{0,4}=(\frac{1+q}{2})^{3}b_{0,1} $. Thus $b_{0,1}=0$. Using  this in (\ref{medje}), we obtain $b_{0,n}=0$ for all $n\neq 3$.\\
Since $b_{0,1}=0$ and  setting $p=3$ in (\ref{formula2}) leads to  $ b_{0,3} =0. $\\
Hence
\begin{equation*}
b_{0,n}=0 \ \forall n\in \ \mathbb{Z}.
\end{equation*}
\textbf{ Case 2: $k=-1$}\\
In (\ref{coh31}), we set $n=-1$ and obtain (for $p\neq 0,1$ )
\begin{eqnarray}
 (1+q^{-1})(\{p+1\}-\{1\})b_{-1,1+p}
-(1+q)(\{p\}-\{1\})b_{-1,p}
&=&0.\label{co1}
\end{eqnarray}
Hence
\begin{equation*}
b_{-1,p}=\frac{1}{q}\frac{\{p+1\}-\{1\}}{\{p\}-\{1\}}b_{-1,1+p}, \ b_{-1,1+p}=q\frac{\{p\}-\{1\}}{\{p+1\}-\{1\}}b_{-1,p}\ \textrm{for}\  p\neq 0,\ 1.
\end{equation*}
The first formula  implies  $b_{-1,p}=q^{p}\frac{\{-1\}}{\{p-1\}}b_{-1,0}$
for $ p\leq 0.$\\
The second formula  implies $b_{-1,p}=q^{p-2}\frac{1}{\{p-1\}}b_{-1,2}$ for $ p\geq2$ .\\
By assumption $b_{-1,1}=0$.

Taking $n=-1,\ p=0$ in (\ref{coh3}), we obtain
\begin{eqnarray}
(1+q^{-1})(\{2\}-\{-1\}) b_{1,0}+(1+q)\{1\}b_{-1,0}
=0.\ \ \ \ \ \ \
\end{eqnarray}
Then 
$
 b_{-1,0}=-q^{-2}\{3\}  b_{1,0}
$. 
We deduce that
\begin{equation*}
 b_{-1,p} = q^{p-3}\frac{\{3\}}{\{p-1\}}  b_{1,0}              ,\ \  \textrm{for}\ p\leq 0 .
\end{equation*}
Taking $n=-1,\ p=1$ in (\ref{coh3}), we obtain
\begin{eqnarray*}
&&-(1+q)(\{2\}-\{-1\})b_{1,0}
+(1+q^{-1})(\{2\}-\{1\})b_{-1,2}
=0.\ \ \ \ \ \
\end{eqnarray*}
Then
$ 
  b_{-1,2} = q^{-1}\{3\} b_{1,0}      
$. 
We deduce that
\begin{equation*}
    b_{-1,p} = q^{p-3}\frac{\{3\}}{\{p-1\}}  b_{1,0}              ,\ \  \textrm{for}\ p\geq 2.
\end{equation*}

\textbf{ Case 3: $k=-2$}\\
Taking $n=-2,\ p=2$ in (\ref{coh3})
\begin{eqnarray*}
&&-(1+q^{3})(\{1\}-\{-2\})b_{-1,2}-(1+q)(\{3\}-\{-2\})b_{1,0}+(1+q^{-2})(\{3\}-\{1\})b_{-2,3}
=0.\ \ \ \ \ \
\end{eqnarray*}
Then
\begin{equation*}
  b_{-2,3}=\frac{1+q^{3}}{1+q^{-2}} \frac{(\{1\}-\{-2\})}{(\{3\}-\{1\})} q^{-1}\{3\} b_{1,0}+\frac{1+q}{1+q^{-2}}\frac{\{3\}-\{-2\}}{\{3\}-\{1\}} b_{1,0}.
\end{equation*}
Taking $n=-2,\ p=1$ in (\ref{coh3}), we obtain
\begin{eqnarray*}
&&(1+q^{-2})(\{2\}-\{1\})b_{-2,2}+(1+q)\{1\}b_{-2,1}
=0.\ \ \ \ \ \ \ \
\end{eqnarray*}
Then
\begin{equation}
 b_{-2,2}=\frac{1+q}{1+q^{-2}}\frac{1-q}{q^2-q} b_{-2,1}= -q\frac{1+q}{1+q^{2}}b_{-2,1}.\label{12juin}
\end{equation}
Taking $n=-2,\ p=0$ in (\ref{coh3}) leads to
\begin{eqnarray*}
&&-(1+q)(\{1\}-\{-2\})b_{-1,0}
+(1+q^{-2})(\{2\}-\{-2\}) b_{1,0}-(1+q)(\{-1\}-\{1\})b_{-2,0}
=0.\ \ \ \ \ \ \ \
\end{eqnarray*}
Then
\begin{eqnarray*}
 b_{-2,0}  = -q^{-3}\frac{\{3\}^2}{\{2\}} b_{1,0}-q^{-3}\frac{1+q^2}{1+q}\frac{\{4\}}{\{2\}}b_{1,0}.
\end{eqnarray*}

We plug the value $n=-2$ into (\ref{coh31}) and get for the terms not yet identified as zero
\begin{equation*}
-(1+q^{p+1})(\{1\}-\{-2\})b_{-1,p}+(1+q^{-2})(\{p+1\}-\{1\})b_{-2,1+p}
-(1+q)(\{p-1\}-\{1\})b_{-2,p}
=0\ \textrm{for}\ p\neq 0,\ 2.
\end{equation*}
This yields
\begin{equation*}
 b_{-2,1+p}=\frac{1+q^{p+1}}{1+q^{-2}}\frac{\{1\}-\{-2\}}{\{p+1\}-\{1\}} q^{p-3}\frac{\{3\}}{\{p-1\}}  b_{1,0}  +\frac{1+q}{1+q^{-2}}\frac{\{p-1\}-\{1\}}{\{p+1\}-\{1\}}b_{-2,p},
\end{equation*}
\begin{equation*}
b_{-2,p}
=- \frac{1+q^{p+1}}{1+q}\frac{\{1\}-\{-2\}}{\{p-1\}-\{1\}} q^{p-3}\frac{\{3\}}{\{p-1\}}  b_{1,0}+                             \frac{1+q^{-2}}{1+q}\frac{\{p+1\}-\{1\}}{\{p-1\}-\{1\}}b_{-2,1+p}\ \textrm{for}\ p\neq 0,\ 1,\ 2.
\end{equation*}
By direct calculation, we obtain
\begin{equation*}
b_{-2,p}=\frac{\{3\}^2}{\{p-2\}}q^{p-5}b_{1,0}-(\frac{1+q^2}{1+q})^{2-p}\frac{\{2\}^2}{\{p-2\}\{p-1\}}q^{2p-6}b_{1,0}\  \textrm{for} \ p\neq 0,\ 1,\ 2,\ 3.
\end{equation*}

\textbf{ Case 4: $k=2$}\\
Taking $n=2,\ m=-1,\ p=-2$ in (\ref{cobbbss}), we obtain
$$
b_{2,-2}=q^{-2}\frac{\{3\}^2}{\{2\}}b_{1,0}+q^{-2}\frac{1+q^2 \{3\}}{1+q\{2\}}b_{1,0}.
$$
Now, taking $n=2,\ m=-2,\ p=-2$ in (\ref{cobbbss}), we obtain
\begin{eqnarray*}
&&-(1+q^{-2})(\{-1\}-\{2\})b_{-2,0}
+(1+q^{2})(\{-1\}-\{-2\})b_{2,-4}\\&&
+(1+q^{2})(\{-3\}-\{2\}) b_{-2,-2}
 - (1+q^{-2})(\{1\}-\{-2\})b_{2,-2}=0
.\nonumber
\end{eqnarray*}
Hence
\begin{eqnarray*}
b_{2,-4}   &=& q^{-2}q\frac{q^{3}-1}{1-q}b_{-2,0}-q^{-1}\frac{q^{5}-1}{1-q}b_{-2,-2}+q^{-2}\frac{1-q^{3}}{1-q}b_{2,-2}.
\end{eqnarray*}
Recall that
\begin{eqnarray*}
 b_{-2,0}  &=& \frac{q^{-2}-q}{q-q^{-1}} q^{-2}\{3\}  b_{1,0}+\frac{1+q^{-2}}{1+q}\frac{q^{-2}-q^{2}}{q-q^{-1}}b_{1,0},\\
 &=& q^{-3}\frac{1-q^{3}}{q^{2}-1} \{3\}  b_{1,0}+ q^{-3}\frac{1+q^{2}}{1+q}\frac{1-q^{4}}{q^{2}-1}b_{1,0},
\end{eqnarray*}
\begin{eqnarray*}
      b_{2,-2} &=&q^{-2}\frac{q^{3}-1}{1-q^2}\frac{1-q^3}{1-q} b_{1,0}+q\frac{1+q^2}{1+q}q^{-1} \frac{q^{4}-1}{1-q^{2}}q^{-5}q^{3}\frac{1-q^{3}}{q^{3}-1}b_{1,0}\\
      &=&-q^{-2}\frac{(1-q^3)^2}{(1-q^2)(1-q)}b_{1,0}+q^{-2}\frac{1+q^2}{1+q} \frac{1-q^{4}}{1-q^{2}}b_{1,0},
    \end{eqnarray*}
\begin{equation*}
b_{-2,-2}=\frac{\{3\}^2}{\{-4\}}q^{-7}b_{1,0}-(\frac{1+q^2}{1+q})^{4}\frac{(1+q)^2}{\{-4\}\{-3\}}q^{-10}b_{1,0}.
\end{equation*}
We deduce that
\begin{eqnarray*}
  b_{2,-4}
  &=&2q^{-4}\frac{(1-q^{3})(1-q^{4})^{2}}{(1-q^2)^3}b_{1,0}\\
  &&-q^{-4}\frac{(1-q^5)(1-q^{3})^2}{(1-q)^2(1-q^{4})}b_{1,0}-q^{-4}\frac{(1-q^{5})(1-q^{4})^2}{(1-q)(1+q)^4(1-q^{3})(1-q^{2})}b_{1,0}.\\
\end{eqnarray*}
Taking $n=2,\ m=-1,\ p=-3$ in (\ref{cobbbss}), we obtain
\begin{eqnarray*}
&&-(1+q^{-1})(\{-2\}-\{2\})b_{-1,-1}
+(1+q^{2})(\{-2\}-\{-1\})b_{2,-4}\\
&&+(1+q^{2})(\{-3\}-\{2\}) b_{-1,-3}
 + (1+q^{-1})\{-1\}b_{2,-3}=0
.\nonumber
\end{eqnarray*}
Hence
\begin{eqnarray*}
  b_{2,-3}
  =q^{-3}\frac{q^{4}-1}{q-1}\frac{1-q^3}{q^{2}-1} b_{1,0}-\frac{1+q^2}{1+q}b_{2,-4}-q^{-3}\frac{1+q^2}{q+1}\frac{q^{5}-1}{q-1}\frac{1-q^{3}}{q^{4}-1} b_{1,0}.\\
\end{eqnarray*}
Then
\begin{eqnarray*}
  b_{2,-3} = q^{-3}\frac{\{4\}\{3\}}{\{2\}}b_{1,0}-2q^{-4}\frac{\{3\}\{4\}^3}{\{2\}^5}b_{1,0}+q^{-4}\frac{\{5\}\{3\}}{\{2\}^2}b_{1,0}
  +q^{-4}\frac{\{5\}\{4\}^3}{\{2\}^7\{3\}}b_{1,0}+q^{-3}\frac{\{5\}\{3\}}{\{2\}^2}b_{1,0}.\\
\end{eqnarray*}

Taking $n=2,\ m=-2,\ p=-3$ in (\ref{cobbbss}), we obtain
\begin{eqnarray}
&&-(1+q^{-2})(\{-2\}-\{2\})b_{-2,-1}
+(1+q^{2})(\{-4\}-\{2\}) b_{-2,-3}
 + (1+q^{-2})\{-2\}b_{2,-3}=0
.\nonumber
\end{eqnarray}
Then
\begin{eqnarray*}
   b_{2,-3}  = \frac{\{4\}}{\{2\}}b_{1,0}-  \frac{\{4\}^4}{\{2\}^6\{3\}}  q^{-3}b_{1,0}
   +\frac{\{6\}}{\{2\}}\frac{\{3\}^2}{\{5\}}q^{-3}+\frac{\{6\}\{4\}^4}{\{2\}^9\{5\}} q^{-3}b_{1,0}.\\
\end{eqnarray*}
Comparing to the previous result, we deduce that $ b_{1,0}=0.$ Hence we get 
\begin{eqnarray*}
  &&b_{1,p} = 0 \  \ \  \forall p\in \mathbb{Z}.\\
       && b_{-1,p} = 0\  \ \  \forall p\in \mathbb{Z},\\
      && b_{-2,p} = 0 \  \ \  \forall p\neq 1,\ 2.\\
        &&  b_{2,p} = 0\  \ \  \textrm{for} \ p=-2,\ -4,\ -3.
\end{eqnarray*}
Taking $n=2,\ m=-1$ in (\ref{cobbbss}), we obtain
\begin{eqnarray}\label{deux}
(1+q^{2})(\{p+1\}-\{-1\})b_{2,p-1}
 - (1+q^{-1})(\{p+3\}-\{-1\})b_{2,p}=0
.\nonumber
\end{eqnarray}
Hence
\begin{equation*}
  b_{2,p}
 =
q\frac{1+q^{2}}{1+q}\frac{1-q^{p+2}}{1-q^{p+4}}b_{2,p-1}\ \textrm{for}\ p\neq -4,  \  b_{2,p-1} =q^{-1}\frac{1+q}{1+q^{2}} \frac{1-q^{p+4}}{1-q^{p+2}}b_{2,p}
   \ \textrm{for}\ p\neq -2.
\end{equation*}
The first formula for $p=-2$ implies $ b_{2,-2}=0$ and hence $b_{2,p}=0$ for $p\geq -2$.\\
The second formula for $p=-4$ implies $ b_{2,-5}=0$ and hence $b_{2,p}=0$ for $p\leq -5$.\\

We take $n=-2,\ m=2,$ $p=-1$ in (\ref{cobbbss}) and recall that terms $b_{0,-1},\ b_{2,-3},\ b_{2,-1}$ and $b_{-2,-1}$  are zero. This gives
 $$b_{-2,1}=0.$$
 Then by (\ref{12juin}) we deduce that $$b_{-2,2}=0.$$
\textbf{ Case 5: $k>2$}\\
By (\ref{coh3}) we have
\begin{eqnarray*}
(1+q^{p+1})(\{1\}-\{n\})b_{n+1,p}
&=&(1+q^{n})(\{p+1\}-\{1\})b_{n,1+p}
-(1+q)(\{n+p+1\}-\{1\})b_{n,p}
.
\end{eqnarray*}
As $k\geq 2$ the value $(1+q^{p+1})(\{1\}-\{k\})\neq0$ and we get by induction trivially the statement for $k+1$ then $$b_{k,p}=0\ \forall k>2.$$

\textbf{ Case 6: $k<-2$}\\
Taking $m=-1$ in (\ref{cobbbss}) since $b_{-1,k}=0, \ \forall k\in \mathbb{Z}$ we have
\begin{equation*}
-(1+q^{p+1})(\{-1\}-\{n\})b_{n-1,p}
+(1+q^{n})(\{p+1\}-\{-1\})b_{n,-1+p}
-(1+q^{-1})(\{n+p+1\}-\{-1\})b_{n,p}
=0.\ \ \ \ \ \ \ \
\end{equation*}
So
\begin{equation*}
(1+q^{p+1})(\{-1\}-\{n\})b_{n-1,p}=
(1+q^{n})(\{p+1\}-\{-1\})b_{n,-1+p}
-(1+q^{-1})(\{n+p+1\}-\{-1\})b_{n,p}
.\ \ \ \ \ \ \ \
\end{equation*}
As $k\leq- 2$ then $(1+q^{p+1})(\{-1\}-\{k\})\neq0$.  We get by induction obviously  the statement for $k-1$.\\
Then $$b_{k,p}=0,\ \forall\ k<-2.$$\\
Finally, we obtain   $$b_{k,p}=0,\ \forall k,\ p\in \mathbb{Z}.$$
\end{proof}
\begin{lem}
Let $f$ be an even $2$-cocycle of degree zero such that $f(L_{1},G_{n})=0$ and $f(L_{-1},L_{2})=0$. 
Then the cohomology class of $f$ is trivial on the space $\mathcal{W}_{1}^{q}\times\mathcal{W}_{1}^{q}$.
\end{lem}
\begin{proof}
By the super skew-symmetry we have $c_{p,m} =c_{m,p} $. In (\ref{paird}) we consider $s=0$ and $n=0$.  This gives
\begin{eqnarray}
-(1+q^{p+1})\{m+1\}c_{m,p}-(1+q^{m+1})\{p+1\}c_{m,p}
+2\{m+p\} c_{m,p}=0.
\end{eqnarray}
Hence,
\begin{equation*}
2q^{m+p}(q^2-1)c_{m,p}=0.
\end{equation*}
This implies $$c_{m,p}=0, \ \forall m,\ p\ \in \mathbb{Z}.$$

\end{proof}
The previous lemmas shows  :
\begin{prop}$$H^2_{0,0}=\{0\}.$$
\end{prop}

%
In the sequel we consider the last case of even $2$-cocycle of degree $2$.
\begin{lem}
Let $f$ be an even  $2$-cocycle of degree $2$ such that $$f(L_n,L_1)=f(L_{-1},L_2)=f(L_{1},G_{n})=f(L_{-1},G_{1})=0.$$ 
 Then the cohomology class of $f$ is trivial.
\end{lem}
\begin{proof}
Define $g$ and $h$ as those given in Lemma (\ref{lemma}). Then by (\ref{lemma1}) we have $h(L_n,L_p)=0$ and by (\ref{hlemma2}) we have $h(L_n,G_p)=0$.\\
Letting $s=2$, $p=0$ in (\ref{paird}), we obtain
\begin{eqnarray}
-(1+q)(\{m+1\}-\{n\})c_{n+m,0}-(1+q^{m+1})(\{1\}-\{n\})c_{n,m}
+(1+q^{n})(\{m+2\}-\{n\}) c_{m,0}=0.\nonumber\\
 \label{pairdj}
\end{eqnarray}
Setting $n=1$ and $m=0$ in (\ref{pairdj}), one can deduce $c_{0,0}=0$.\\
Then taking $m=0$ in (\ref{pairdj}), we obtain $c_{n,0}=0$ for $n\neq 1$.\\
Taking $n=1$ and $m=1$ in (\ref{pairdj}), we obtain $c_{1,0}=0$. Then $c_{n,0}=0,\ \forall n\in \mathbb{Z}$.
Using this in (\ref{pairdj}), one has $$c_{n,m}=0,\ \forall n,\ m \in \mathbb{Z}.$$
\end{proof}

\subsection{Second odd cohomology $H^{2}_{1}(\mathcal{W}^{q},\mathcal{W}^{q})$}
Let $H^{2}_{1,s}(\mathcal{W}^{q},\mathcal{W}^{q})$ be the subspace of  second odd cohomology group given by odd $2$-cochains of degree  $s$. 
 Let $f$ be an odd $2$-cocycle of degree $s.$ We can assume that
\begin{equation*}
   f(L_n,L_p)=a_{n,p}G_{s+n+p},\  f(L_n,G_p)=b_{n,p}L_{s+n+p}\ \textrm{ and} \ f(G_n,G_p)=c_{n,p}G_{s+n+p}.\label{impair}
\end{equation*}
Thus, by (\ref{pair}), we have
\begin{eqnarray}\label{1odd}
&& -(1+q^{p})(\{m\}-\{n\})a_{n+m,p}+ (1+q^{m})(\{p\}-\{n\})a_{n+p,m}+(1+q^{n})(\{p\}-\{m\})a_{n,m+p} \nonumber\\
&&+(1+q^{n})(\{m+p+s+1\}-\{n\})a_{m,p}-(1+q^{m})(\{n+p+s+1\}-\{m\})a_{n,p} \nonumber\\
&&+(1+q^{p})(\{n+m+s+1\}-\{p\})a_{n,m} =0.
\end{eqnarray}
By (\ref{pair2}), we obtain
\begin{eqnarray}\label{2odd}
&& -(1+q^{p+1})(\{m\}-\{n\})b_{n+m,p}- (1+q^{m})(\{p+1\}-\{n\})b_{m,n+p}+(1+q^{n})(\{p+1\}-\{m\})b_{n,m+p} \nonumber\\
&&+(1+q^{n})(\{m+p+s\}-\{n\})b_{m,p}-(1+q^{m})(\{n+p+s\}-\{m\})b_{n,p}  =0.
\end{eqnarray}
By (\ref{pair3}), we obtain
\begin{eqnarray}\label{oddjeudi}
 &&-(1+q^{p+1})(\{m+1\}-\{n\})c_{n+m,p}- (1+q^{m+1})(\{p+1\}-\{n\})c_{n+p,m}\nonumber\\
 && +(1+q^{n})(\{m+p+s+1\}-\{n\})c_{m,p}
 -(1+q^{m+1})(\{m+1\}-\{n+p+s\})b_{n,p}\\
 && -(1+q^{p+1})(\{p+1\}-\{n+m+s\})b_{n,m}=0.\nonumber
\end{eqnarray}

\begin{prop}\label{lemme1}\label{20juin}
If $s\neq 1,-1$, the subspace $H^{2}_{1,s}(\mathcal{W}^{q},\mathcal{W}^{q})$ is
trivial.
\end{prop}
\begin{proof}
We define an endomorphism $g$ of $\mathcal{W}^{q}$ by
\begin{equation*}
   g(L_{p})=\frac{1}{q^{p}\{s+1\}}f(L_{0},L_{p})  \textrm{ and} \ g(G_{p})=\frac{1}{q^{p+1}\{s-1\}}f(L_{0},G_{p}).
\end{equation*}
By (\ref{1cochain}) and (\ref{crochet1}) we have
\begin{equation*}
\delta^{1}(g)(L_{0},L_{p}) =-\{p\}g(L_{p})+\{p+s+1\}g(L_{p}),
\end{equation*}
so
\begin{equation*}
\delta^{1}(g)(L_{0},L_{p})=q^{p}\{s+1\}g(L_{p}).
\end{equation*}
We define a $2$-cocycle $h$ by
\begin{equation*}
h=f-\delta ^{1} (g).
\end{equation*}
Therefore
 \begin{equation}
h(L_{0},L_{p})=f(L_{0},L_{p})-\delta^{1}(g)(L_{0},L_{p})
=q^{p}\{s+1\}g(L_{p})-q^{p}\{s+1\}g(L_{p})=0.\label{un}
\end{equation}
 Taking $m=0$ in (\ref{pair}), with (\ref{crochet1})and (\ref{alpha1}),    we obtain
\begin{eqnarray}
0&=&(1+q^{p})\{n\}f(L_{n},L_{p})+2(\{p\}-\{n\})f(L_{n+p},L_{0})+(1+q^{n})\{p\}f(L_{n}, L_{p})\nonumber\\
&&+(1+q^{n})[L_{n},f(L_{0},L_{p})]-2[L_{0},f(L_{n},L_{p})]+(1+q^{p})[L_{p},f(L_{n},L_{0})] \label{marss}.
\end{eqnarray}
Since $h$ is a $2$-cocycle,  we can replace $f$ by $h$ in (\ref{marss}), and using (\ref{un}) we obtain
\begin{eqnarray*}
0=(1+q^{p})\{n\}h(L_{n},L_{p})+(1+q^{n})\{p\}h(L_{n}, L_{p})-2\{n+p+s+1\}h(L_{n},L_{p}).
\end{eqnarray*}
From this, using  the fact $s\neq -1$, we obtain
 $$h(L_{n},L_{p})=0\ \forall \ n,p\in \mathbb{Z}.$$
 Since  $g(G_{p})=\frac{1}{q^{p+1}\{s-1\}}f(L_{0},G_{p})$. By (\ref{1cochain2}) and (\ref{crochet2}) we have
\begin{equation*}\label{}
  \delta^{1}(g)(L_{0},G_{p})=-\{p+1\}g(G_{p})+\{p+s\}g(G_{p})=q^{p+1}\{s-1\}g(G_{p}).
\end{equation*}
Then
\begin{equation*}
h(L_{0},G_{p})=f(L_{0},G_{p})-\delta^{1}(g)(L_{0},G_{p})
=0.
\end{equation*}
Taking $m=0$ in (\ref{pair2}),\ by (\ref{crochet1}), (\ref{crochet2}), (\ref{alpha1}) and(\ref{alpha2}) we obtain
\begin{eqnarray*}
&&(1+q^{p+1})\{n\}f(L_{n},G_{p})+2(\{p+1\}-\{n\})f( G_{p+n},L_{0})+(1+q^{n})\{p+1\}f(L_{n}, G_{p})\\
&&+(1+q^{n})[L_{n},f(L_{0},G_{p})]-2[L_{0},f(L_{n},G_{p})]+(-1)^{|f|}(1+q^{p+1})[G_{p},f(L_{n},L_{0})]=0.
\end{eqnarray*}
Since $h$ is a $2$-cocycle and $h(L_{0},G_{p})=h(L_{0},L_{n})=0$, we can deduce that
\begin{eqnarray*}
(1+q^{p+1})\{n\}h(L_{n},G_{p})+(1+q^{n})\{p+1\}h(L_{n},G_{p})-2\{p+n+s\}h(L_{n},G_{p})=0,
\end{eqnarray*}
which implies that
 $h(L_{n},G_{p})=0$ under the condition $s\neq1$.\\
Taking $n=0$ in (\ref{pair3}), since $[G_{m},G_{p}]=h(L_{n},G_{p})=h(L_{n},L_{p})=0$ and $h$ is a $2$-cocycle we have
\begin{eqnarray*}
-h([L_{0},G_{m}],\alpha(G_{p}))-h([L_{0},G_{p}],\alpha(G_{m}))+[\alpha(L_{0}),h(G_{m},G_{p})]=0.
\end{eqnarray*}
Then
\begin{equation*}
- (1+q^{p+1})\{m+1\} h( G_{m},G_{p}) - (1+q^{m+1})\{p+1\}h(G_{p},G_{m})+2\{m+p+s+1\}h(G_{m},G_{p})=0.
\end{equation*}
So, under the condition $s\neq 1$,  $h( G_{m},G_{p})=0$.\\
We deduce that $h\equiv 0$. Hence $f$ is a coboundary.
\end{proof}
\begin{lem}
Let $f$  be an  odd $2$-cocycle of degree one and  $g$ be an odd  endomorphism  of $\mathcal{W}^{q}$.\\ If $h=f-\delta^{1}(g)$
then for all $n\in \mathbb{Z},\ m\in \mathbb{Z}^{*}$, we have $$h(L_{n},L_1)=0,\ h(L_1,G_m)=0,\ h(L_{2},L_{-1})=0, \ \textrm{and} \ h(L_{-1},G_{1})=0.$$
\end{lem}

\begin{proof}
  Since $f$ is an   odd $2$-cocycle of degree $1$, we can assume that $f(L_{n},L_{m})=f_{n,m}G_{n+m+1}$ and  $f(L_{n},G_{m})=f_{n,m}^{'}L_{n+m+1}$ .\\
Let  $(a_{n})_{n\in\mathbb{Z}}$  be  the sequence  given recursively by
\begin{eqnarray*}
   && a_{-1}=0,\\
    && a_{1}=\frac{1}{q^2-1}f_{-1,1}+\frac{1}{q^{2}-q^3}f_{0,1},\\
&&a_{n}=\frac{1}{\{n+2\}-\{1\}}\big(-f_{n,1}+(\{3\}-\{n\})a_{1}+(\{n\}-\{1\})a_{n+1}\big)  \ \ \forall n<-1,\\
    &&a_0=\frac{1}{q-1}f_{-1,1}+\frac{1+q^{-2}}{1-q}f_{0,1},\\
    && a_{2}=\frac{1-q}{q^{4}-q^{-1}}f_{2,-1}+\frac{q^{3}-1}{(q^{5}-1)(q^2-1)}f_{-1,1}-\frac{q^{3}-1}{(q^{5}-1)(q^{3}-q^2)}f_{0,1},
\\
&&
    a_{n+1}=\frac{1}{\{n\}-\{1\}}\big(f_{n,1}-(\{3\}-\{n\})a_{1}+(\{n+2\}-\{1\})a_{n}\big),\ \ \ \forall n>1.
\end{eqnarray*}
 Let  $(b_{m})_{m\in\mathbb{Z}}$  be  the sequence  given recursively by
 \begin{eqnarray*}
  && b_{1}=\frac{1}{\{2\}-\{-1\}}f'_{-1,1},\\
  &&b_0=0,\\
&&b_{m}=b_{m+1}+\frac{1}{\{m+1\}-\{1\}}f_{1,m}^{'}  \ \ \forall m<0,
   \\
&&b_{m+1}=b_{m}+\frac{1}{\{1\}-\{m+1\}} f_{1,m}'  \ \ \forall m>0.
   \end{eqnarray*}
Let $g$ be an odd   endomorphism  of  degree $1$ of $\mathcal{W}^{q}$ given recursively by
$g(L_{n})=a_{n}G_{n+1}$ and $g(G_{n})=b_{n}L_{n+1}.$\\
By (\ref{1cochain}) and (\ref{1cochain2}) we have  recursively
 \begin{eqnarray*}
\delta^{1}(g)(L_{n},L_{m})=\bigg((\{n\}-\{m\})a_{n+m}+(\{m+2\}-\{n\})a_{m}-(\{n+2\}-\{m\})a_{n}\bigg)G_{n+m+1},
  \end{eqnarray*}
  and
   \begin{eqnarray}
\delta^{1}(g)(L_{n},G_{m})=-(\{m+1\}-\{n\})(b_{n+m}-b_{m})L_{n+m+1}.
  \end{eqnarray}
   If $h=f-\delta^{1}(g)$ we have
  \begin{eqnarray*}
&& h_{n,1}=f_{n,1}-\big((\{n\}-\{1\})a_{n+1}+(\{3\}-\{n\})a_{1}-(\{n+2\}-\{1\})a_{n}\big)=0\ \forall n\in \mathbb{Z},\\
&& h_{2,-1}=f_{2,-1}-\big((\{2\}-\{-1\})a_{1}+(\{1\}-(\{2\})a_{-1}-(\{4\}-\{-1\})a_{2}\big)=0,\\
&& h_{1,m}^{'}=f_{1,m}^{'}+(\{m+1\}-\{1\})(b_{m+1}-b_{m})=0\ \forall m\in \mathbb{Z}^{*},\\
&& h_{-1,1}^{'}=f_{-1,1}^{'}+\{2\}(b_{0}-b_{1})=0 .
     \end{eqnarray*}

\end{proof}
\begin{lem}
Let $f$ be an odd $2$-cocycle of degree one such that $f(L_{n},L_{1})=0$ and $f(L_{-1},L_{2})=0$. 
Then the cohomology class of  $f$ is trivial   on the space $\mathcal{W}_{0}^{q}\times \mathcal{W}_{0}^{q}$.
\end{lem}
\begin{proof}
We consider the linear map $g$ defined by $g(L_n)=\frac{1}{q^n\{2\}}$ and $\displaystyle g(G_n)=0$ and argue as in proof of Lemma \ref{lemme1} ,    

\end{proof}

\begin{lem}
Let $f$ be an odd $2$-cocycle of degree one  such that  $f(L_{-1},G_{1})=f(L_{2},L_{-1})=0$ and $f(L_{1},L_n)=f(L_{1},G_{m})=0,\ \forall n\in \mathbb{Z},\ m \in \mathbb{Z}^*$.
Then the cohomology class of $f$ is trivial  on the space $\mathcal{W}_{0}^{q}\times\mathcal{W}_{1}^{q}$.
\end{lem}
\begin{proof}
Let $f$ be an odd
$2$-cocycle  of degree one. Then using (\ref{2odd}) we obtain exactly the same equation as  \eqref{cobbbss}. Therefore the proof is similar to Lemma \ref{lemma3.4}.

\end{proof}

\begin{lem}
Let $f$ be an odd $2$-cocycle of degree one such that $f(L_{2},L_{-1})=0$,   $f(L_{-1},G_{1})=0$,    $f(L_{1},L_{n})=0\ \forall n \in \mathbb{Z}$ and       $f(L_{1},G_{m})=0,\ \forall m\in \mathbb{Z}^{*}$.
Then the cohomology class of $f$ is trivial.
\end{lem}
\begin{proof}
Since $f$ is an odd
$2$-cocycle  of degree one, we can assume that
\begin{eqnarray*}
f(G_{m},G_{n})=f(G_{n},G_{m})&=&c_{n,m}G_{n+m+1}.
\end{eqnarray*}

Since $b_{n,k}=0,\ s=1$, then (\ref{pair3}) can be written 
\begin{eqnarray}\label{soddjeudi}
&& -(1+q^{p+1})(\{m+1\}-\{n\})c_{n+m,p}- (1+q^{m+1})(\{p+1\}-\{n\})c_{n+p,m}\nonumber\\&&+(1+q^{n})(\{m+p+2\}-\{n\})c_{m,p} =0.
\end{eqnarray}
Therefore, if $p=0$, we obtain
\begin{equation}\label{soddc}
    -(1+q)(\{m+1\}-\{n\})c_{n+m,0}-(1+q^{m+1})(\{1\}-\{n\})c_{n,m}+(1+q^{n})(\{m+2\}-\{n\})c_{m,0}=0.
\end{equation}
Taking $n=1$, $m=-1$ in (\ref{soddc}), we obtain $c_{0,0}=0$.\\
Taking  $m=0$ in (\ref{soddc}), we obtain (with $c_{0,0}=0$)
$
    c_{n,0}=0, \ \forall n\neq 1.
$ 
Taking $n=1$,  $m=1$ in (\ref{soddc}), we obtain (with $c_{2,0}=0$ )
$
    c_{1,0}=0.
$ 
We deduce that
\begin{equation}\label{}
 c_{n,0}=0\   \forall  n\in \mathbb{Z}.
\end{equation}

 Using this in (\ref{soddc}), we get
 $c_{n,m}=0,\  \forall (n,m)\neq (1,1).$\\
Taking $n=2$,  $m=1$, $p=1$  in (\ref{soddjeudi}), we obtain $c_{1,1}=0$.

\end{proof}
Finally,
\begin{prop}
$$ H^2_{1,1}=\{ 0\}.$$
\end{prop}
Now we consider odd cocycles of degree $-1$.
\begin{lem}
Let $f$  be an  odd $2$-cocycle of degree $-1$ and  $g$ be an odd  endomorphism  of $\mathcal{W}^{q}$.\\ If $h=f-\delta^{1}(g)$
then $h(L_{1},L_n)=0,\ h(L_1,G_m)=0,\ h(L_{-1},G_1)=0\  \textrm{and}\ h(L_{2},L_{-1})=0.$ \
\end{lem}
\begin{proof}
  Since $f$ is an   odd $2$-cocycle of degree $-1$, we  assume that $f(L_{n},L_{m})=f_{n,m}G_{n+m-1}$ and  $f(L_{n},G_{m})=f_{n,m}^{'}L_{n+m-1}$ .\\
Let  $(a_{n})_{n\in\mathbb{Z}}$  be  the sequence  given recursively by

\begin{eqnarray*}
 &&a_{0}=-f_{1,0},\\
 &&a_{n}=a_{n+1}-\frac{1}{\{1\}-\{n\}}f_{1,n} \ \ \forall n\leq0,\\
 && a_{1}=0\\
  && a_{2}=-\frac{1}{\{2\}-\{-1\}}f_{2,-1}-a_{-1},\\
   && a_{n+1}=\frac{1}{\{1\}-\{n\}}f_{1,n}+a_{n},\ \ \ \forall n>1.
\end{eqnarray*}
 Let  $(b_{m})_{m\in\mathbb{Z}}$  be  the sequence  given recursively by
 \begin{eqnarray*}
  &&b_0=-\frac{q}{1+q}f'_{1,0},\\
&&b_{m}=\frac{\{m+1\}-\{1\}}{\{m-1\}-\{1\}}b_{m+1} +\frac{1}{\{m-1\}-\{1\}}f_{1,m}^{'} \ \ \forall m<0,\\
    &&b_{1}=qf_{-1,1}'-q\frac{\{3\}}{\{2\}}f'_{1,0},\\
&&b_{m+1}=\frac{\{1\}-\{m-1\}}{\{1\}-\{m+1\}}b_{m}+\frac{1}{\{1\}-\{m+1\}} f_{1,m}'  \ \ \forall m\geq1.
   \end{eqnarray*}
Let $g$ be an  odd   endomorphism of  degree $-1$  of $\mathcal{W}^{q}$ given recursively by
$g(L_{n})=a_{n}G_{n-1}$ and $g(G_{n})=b_{n}L_{n-1}.$ 
By (\ref{1cochain}) and (\ref{1cochain2}) we have  recursively
 \begin{eqnarray*}
\delta^{1}(g)(L_{n},L_{m})=(\{n\}-\{m\})\bigg(a_{n+m}-a_{m}-a_{n}\bigg)G_{n+m-1},
  \end{eqnarray*}
  and
   \begin{eqnarray}
\delta^{1}(g)(L_{n},G_{m})=\bigg((\{n\}-\{m+1\})b_{n+m}+(\{m-1\}-\{n\})b_{m})\bigg)L_{n+m-1}.
  \end{eqnarray}
   If $h=f-\delta^{1}(g)$ we have
  \begin{eqnarray*}
&& h_{1,n}=f_{1,n}-(\{1\}-\{n\})(a_{n+1}-a_{n})=0\ \forall n\in \mathbb{Z},\\
&&h_{2,-1}=f_{2,-1}-(\{2\}-\{-1\})(a_{1}-a_{2}-a_{-1})=0.\\
&&h_{1,m}^{'}=f_{1,m}^{'}-\bigg((\{1\}-\{m+1\})b_{m+1}+(\{m-1\}-\{1\})b_{m}\bigg)=0\ \forall m\in \mathbb{Z}.\\
&&h_{-1,1}^{'}=f_{-1,1}^{'}-(\{- 1\}-\{2\})b_{0}+\{-1\}b_1=0\ .\\
     \end{eqnarray*}

\end{proof}

\begin{lem}
Let $f$ be an odd $2$-cocycle of degree $-1$ such that $f(L_{n},L_{1})=0$ and $f(L_{2},L_{-1})=0$.
Then the cohomology class of $f$ is trivial.
\end{lem}
\begin{proof}
Let $f$ be an odd
$2$-cocycle  of degree $-1$. Then plugging $s=-1$ in (\ref{1odd}) leads to equation \eqref{formul 1}. Therefore the proof goes the same as Lemma  \ref{lemma3.3}.
\end{proof}
\begin{lem}
Let $f$ be an odd $2$-cocycle of degree $-1$ such that $f(L_{1},G_{n})=0$, $f(L_{-1},G_{1})=0$ and $f(L_{2},L_{-1})=0$. 
Then the cohomology class of $f$ is  trivial   on the space $\mathcal{W}_{0}^{q}\times\mathcal{W}_{1}^{q}$.
\end{lem}
\begin{proof}
We consider the linear map $g$ defined by $g(L_n)=0$ and $\displaystyle g(G_p)=\frac{1}{q^{p+1}\{-2\}}f(L_0,L_p)$ and argue as in the proof of Lemma \ref{20juin} .  
\end{proof}

Finally,
\begin{prop}
$$ H^2_{1,-1}=\{ 0\}$$
\end{prop}

\end{document}